\documentclass[11pt]{article}
\usepackage{amsfonts,amsmath,amssymb,amsthm}
\usepackage{epsfig}
\usepackage{amsmath}
\usepackage{amssymb}
\usepackage{latexsym}
\usepackage{amsfonts}
\usepackage{amsthm}
\usepackage[compress]{cite} 
\usepackage{enumerate}  
\usepackage{graphicx}
\usepackage{algorithm} 
\usepackage{algorithmic} 
\usepackage{subfigure} 
\usepackage{arydshln}  
\usepackage{color}

\theoremstyle{definition}

\theoremstyle{plain}
\newtheorem{theorem}{\bf Theorem}[section]
\newtheorem{proposition}[theorem]{\bf Proposition}
\newtheorem{claim}{\bf Claim}[section]
\newtheorem{lemma}[theorem]{\bf Lemma}
\newtheorem{corollary}[theorem]{\bf Corollary}

\numberwithin{figure}{section}
\theoremstyle{observation}
\newtheorem{observation}[theorem]{\bf Observation}

\title{The restrained double Roman domination\\ and graph operations}
\author{\small {Zhipeng Gao$^{1}$, Changqing Xi$^{1}$, Jun Yue$^{2}$}\\[5mm]
{\small $^1$Center for Combinatorics and LPMC, Nankai University, Tianjin, China}\\
{\small $^2$ School of Mathematics and Statistics}\\
{\small Shandong Normal University, Jinan, Shandong 250358, P.R. China}\\
{\small Emails: 1120180009@mail.nankai.edu.cn; xcqmath@163.com; yuejun06@126.com}\\
}
\date{}

\begin{document}\maketitle
\begin{abstract}
Let $G=(V(G),E(G))$ be a simple graph. A restrained double Roman dominating function (RDRD-function) of $G$ is a function $f: V(G) \rightarrow \{0,1,2,3\}$ satisfying the following properties: if $f(v)=0$, then the vertex $v$ has at least two neighbours assigned 2 under $f$ or one neighbour $u$ with $f(u)=3$; and if $f(v)=1$, then the vertex $v$ must have one neighbor $u$ with $f(u) \geq 2$; the induced graph by vertices assigned 0 under $f$ contains no isolated vertex. The weight of a RDRD-function $f$ is the sum $f(V)=\sum_{v \in V(G)} f(v)$, and the minimum weight of a RDRD-function on $G$ is the restrained double Roman domination number (RDRD-number) of $G$, denoted by $\gamma_{rdR}(G)$. In this paper, we first prove that the problem of computing RDRD-number is NP-hard even for chordal graphs. And then we study the impact of some graph operations, such as strong product, cardinal product and corona with a graph, on restrained double Roman domination number.

\medskip\noindent
\textbf{AMS Math.\ Subj.\ Class.\ (2010)}: 05C69, 05C07
\end{abstract}

\section{Introduction}
Let $G=(V(G),E(G))$ be a simple graph with vertex set $V(G)$ and edge set $E(G)$. We always denote $|V(G)|$ and $|E(G)|$ by order $n$ and size $m$, respectively. 
For any vertex $v\in V(G)$, the \emph{open neighborhood} of $v$ is the set $\{u|(u,v)\in E\}$, denoted by $N(v)$. The \emph{degree} of a vertex $v\in V(G)$ is $d(v)=|N(v)|$ stands for the number of vertices in $N(v)$. The minimum degree (resp., maximum degree) among the vertices of $G$ is denoted by $\delta(G)$ (resp., $\Delta(G)$). 
A vertex $v$ with $d(v)=1$ is a \emph{leaf}, and its neighbor is called a \emph{support vertex}. If there are at least two leaves in the neighborhood of a vertex $u$, then $u$ is called a \emph{strong support vertex}. A \emph{wounded spider} $ws(1,n,t)$ is a tree by subdividing $t$ edges of the star $K_{1,n}$, where $0\le t\le n-1$. For more graph theoretic terminology and notations not given here, the reader is referred to \cite{West}.

A subset $D$ of $V(G)$ is a \emph{domination set} if every vertex outside $D$ has a neighbor in $D$. The \emph{domination number} of $G$, denoted by $\gamma(G)$, is the minimum cardinality of a domination set of $G$. A domination set $D$ of $G$ is a \emph{restrained domination set} if every vertex outside $D$ is adjacent to another vertex in $V(G)\setminus D$.
The \emph{restrained domination number} of $G$, denoted by $\gamma_{r}(G)$, is the minimum cardinality of a restrained domination set in $G$. The concept of the restrained domination was formally introduced in \cite{Domke}. Since then, the variants of restrained domination have already been studied, the details refer to \cite{Chen11,RoushiniLeely1,RoushiniLeely2,Samadi}.


Double Roman domination is a stronger version of Roman domination, which was introduced in \cite{Beeler}. A \emph{double Roman dominating function} (for short, DRD-function) of $G$ is a function $f: V(G) \rightarrow \{0,1,2,3\}$ satisfying the following conditions: if $f(v)=0$, then the vertex $v$ has at least two neighbours assigned 2 under $f$ or one neighbour $u$ with $f(u)=3$; and if $f(v)=1$, then the vertex $v$ must have one neighbor $u$ with $f(u) \geq 2$. At the same time, the DRD-function satisfying the subgraph induced by vertices assigned 0 under $f$ contains no isolated vertex, we call it a \emph{restrained double Roman dominating function} (RDRD-function) of $G$. Note that if any vertex $v\in V(G)$ satisfying the above conditions is said to be \emph{restrained double Roman dominated}. The weight of an RDRD-function $f$ is the sum $f(V)=\sum_{v \in V(H)} f(v)$, and the minimum weight of an RDRD-function on $G$ is the \emph{restrained double Roman domination number} (RDRD-number) of $G$, denoted by $\gamma_{rdR}(G)$. For the sake of convenience, an RDRD-function $f$ of a graph $G$ with weight $\gamma_{rdR}(G)$ is called a $\gamma_{rdR}(G)$-function. Since a restrained double Roman dominating function induces the ordered partition of $V(G)$ with $V_i=\{v\in V(G): f(v)=i\}$ for $i=\{0,1,2,3\}$, we are allowed to write $f=(V_0,V_1,V_2,V_3)$.

The concept of restrained double Roman domination was recently introduced by Mojdeha, Masoumib and Volkmann in \cite{Mojdeha2021}. They proved that computing the value of RDRD-number is NP-Complete for general graphs. Moreover, an upper bound on RDRD-number of a connected graph $G$ was given, and characterization of the graphs achieving the bound was presented. In addition, they compared the parameters between RDRD-number and other domination numbers. And then, Samadi et al. \cite{Samadi1} proved that the problem of computing RDRD-number is NP-hard even for planar graphs, and investigated its relationships with some well-known parameters such as the restrained domination number and the domination number for general graphs, respectively. They also gave the characterization of graphs with small RDRD-numbers, which was independently given by Xi and Yue in \cite{Xi}. In \cite{Xi}, the authors gave a linear time algorithm for computing the RDRD-number of a cograph.

The graph product is vital as the binary operation on graphs. Specially, it is an operation that takes two graphs $G_1$ and $G_2$ and produces a new graph with desired properties. By applying the graph product operations, we can enlarge the smaller graphs, and can also identify and decompose the whole graph with the help of factor graphs to understand the structure and topological properties of the whole. Graph products have various applications such as
computer and communication networks, management of multiprocessors and automata theory\cite{cardinal1}. Now we give some definitions of the graph products as follows. The cardinal product of two graphs $G$ and $H$ (called sometimes direct product, cross product or Kronecker product), denoted by $G\times H$, is the graph with $V(G\times H)=V(G)\times V(H)$ and $(u,v)(u',v')\in E(G\times H)$ if $uu'\in E(G)$ and $vv'\in E(H)$. The strong product of two graphs $G$ and $H$, denoted by $G\boxtimes H$, is the graph with $V(G\boxtimes H)=V(G)\times V(H)$ and $(u,v)(u',v')\in E(G\boxtimes H)$ if, either ($uu'\in E(G)$ and $v=v'$), $(vv'\in E(H)$ and $u=u'$) or ($uu'\in E(G)$ and $vv'\in E(H)$). The corona of two graphs $G$ and $H$, denoted by $G\odot H$, is the graph formed by taking one copy of $G$, $|V(G)|$ copies of $H$, and then joining $i^{th}$ vertex of $V(G)$ to every vertex in the $i^{th}$ copy of $H$. In particular, $G\odot K_1$ is the graph constructed from one copy of $G$, and for each vertex $v\in V(G)$, a new vertex $v'$ as a leaf added.


Vizing's conjecture on Cartesian product graphs has become one of the most interesting problems of domination in graphs since it was introduced \cite{Vizing1963}, and it has motivated a large number of scholars to focus on the domination problems of Cartesian product graphs. Another graph product that provides an interesting and non-trivial problems on domination is the cardinal product. A Vizing-like conjecture for the cardinal product was proposed in \cite{cardinal4}, namely $\gamma(G\times H)\ge\gamma(G)\gamma(H)$. Subsequently, some researchers gave a negative answer to the Vizing-like conjecture. This further reflects the difficulty of studying the cardinal product of graphs. As Sandi Klav\v{z}ar stated in \cite{cardinal4}, ``Although the direct product(cardinal product) of graphs is the most natural graph product, it is also the most difficult and unpredictable among standard graph products".

Several scholars have discussed the properties and the domination number of $P_n\times P_m$. Klobu\v{c}ar \cite{cardinalpp1,cardinalpp2} determined the exact values of the domination number of $P_n\times P_m$ for $2\le n\le 6$, and also presented the bounds of the domination numbers of $P_7\times P_m$ and $P_8\times P_m$. Ch\'{e}rifi et al. \cite{cardinalpp3} designed an algorithm to get the exact values of the domination number of $P_9\times P_m$ for $m\ge8$. Roman domination in Cartesian product graphs and strong product graphs were studied by Yero et al. \cite{strong}. Recently, the total Roman domination number of the graph $G \times H$ was studied by Mart\v{i}nez et al. \cite{cardinal3}. And they also studied the cardinal product graphs $G\times H$ for various specific graphs. Klobu\v{c}ar et al. \cite{cardinal2} focused on the double Roman domination numbers on cardinal products of graphs, and they obtained the exact values of $\gamma_{dR}(P_2\times G)$ for some special graphs $G$. Moreover, the exact values of $\gamma_{dR}(P_k\times P_n)$ were determined, where $k\in\{2,3,4\}$. In addition, they presented the upper and lower bounds for $\gamma_{dR}(P_5\times P_n)$ and $\gamma_{dR}(P_6\times P_n)$. In \cite{Anu2019}, the impact of various graph operations on the double Roman domination number was studied, including cartesian product, addition of twins and corona with a graph. These results have laid the foundation for our research work.

In this paper, we will mainly consider the RDRD-number of graphs. In section 2, we prove that the problem of computing RDRD-number is NP-hard even for chordal graphs. In section 3, we determine the exact values of RDRD-numbers of graphs $P_2\boxtimes P_m$ and $P_3 \boxtimes P_m$, where $m \geq 2$. In section 4, we present an upper bound and a lower bound of RDRD-number of $G\times H$ for any two graphs $G$ and $H$. And then, we obtain the exact values of RDRD-number of $C_3\times C_m$ and $P_2\times G$ for some special graph $G$, where $m \geq 3$. In the last section, we obtain the exact value of RDRD-number of $G\odot H$ when $(H\ncong K_1)$, and present the sharp bounds of $G\odot K_1$. Moreover, the exact values of $\gamma_{rdR}(G\odot K_1)$ are determined, where $G$ is a complete graph, a cycle, a path or a complete bipartite graph.

\section{Complexity Result}
In this section, we will prove that the decision problem associated with RDRD-number is NP-Complete even for chordal graphs. That is stated in the following decision problem, to which we shall refer as RDRD-Problem.

\textbf{RDRD-Problem}:

\textbf{Instance}:Given a chordal graph $G=(V,E)$.

\textbf{Question}: Is there a RDRD-function for $G$ whose weight is at most $k$?

We will show that this problem is NP-Complete by reducing the well-known NP-Complete problem, Exact-3-Cover(X3C), to RDRD-Problem.

\textbf{Exact-3-Cover(X3C)}:

\textbf{Instance}: A finite set $X$ with $|X|=3q$ and a collection $\mathcal{C}$ of 3-element subsets of $X$.

\textbf{Question}: Is there a subcollection $C'$ of $C$ such that every element of $X$ appears in exactly one element of $C'$?

The ideal to prove the following theorem is similar to the proof of Theorem 5 in \cite{ACS}.
\begin{theorem}
RDRD-Problem is NP-Complete for chordal graphs.
\end{theorem}

\begin{proof}
RDRD-Problem belongs to $\mathcal{NP}$-class, since it is a polynomial time to decide that a function is a RDRD-function with weight at most $k$. Now we show how to transform any instance of X3C into an instance in RDRD-Problem satisfying that one of them has a solution if and only if the other has a solution.

Let $X=\{x_1,x_2,\ldots,x_{3q}\}$ and $C=\{C_1,C_2,\ldots,C_t\}$ be any arbitrary instance of X3C. We construct a chordal graph $G$ as follows. First, we create a bipartite graph with vertices $x_i$ respect to each $x_i \in X$ and $c_i$ respect to each $C_i \in C$, whose edge $x_ic_j$ if and only if $x_j \in C_j$, for $i \in [3q]$ and $j \in [t]$. Second, in the so far constructed graph, add a pendant vertex $y_i$ to each vertex $x_i$, create a new vertex $c'_j$ which has the same neighbors of $c_j$, where $i \in [3q]$ and $j \in [t]$. And then make $B:=\{c_1,c_2,\ldots,c_t,c'_1,c'_2,\ldots,c'_t\}$ form a clique. Third, add new five vertices $z,z_1,z_2,z_3,z_4$ such that $z$ is adjacent to all vertices in $B$, $z,z_1,z_2$ and $z,z_3,z_4$ are form triangles. It is easy to check that this graph is a chordal graph. An example can be seen in Figure \ref{fig00}, where $X=\{x_1,x_2,\ldots,x_6\}$ and $C=\{\{x_1,x_3,x_5\},\{x_2,x_4,x_6\}\}$.


\begin{figure}[H]
\centering
\includegraphics[scale=0.46]{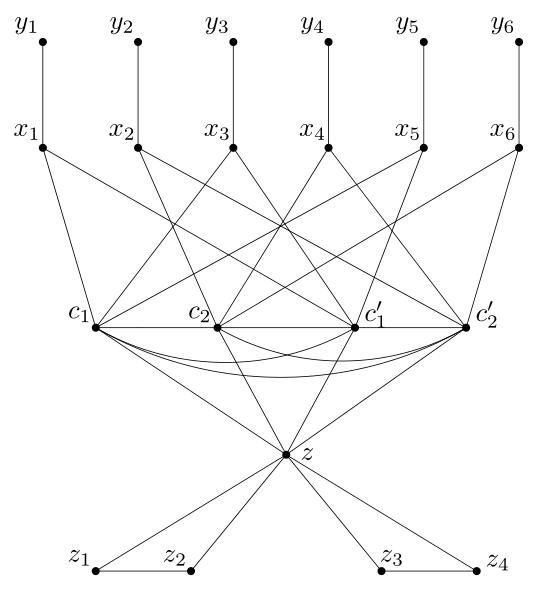}
\caption{NP-Completeness for chordal graphs.}
\label{fig00}
\end{figure}

Let $k=8q+3$. Suppose that the instance $X$, $C$ of X3C has a solution $C_S \subset C$. Let $Y=\{y_1,y_2,\ldots,y_{3q}\}$. We give a RDRD-function $f$ on $G$ as follows.
\begin{equation}{
f(v)=\left\{
\begin{array}{llll} \vspace{0.1cm}
	3,\quad & v = z,\\
	2,\quad &v \in Y\cup C_S,\\
	0,\quad &oterwise.
\end{array}\notag
\right.}
\end{equation}
Clearly, $w(f)= 2\times3q+2\times q+3=8q+3=k$.
	
Conversely, suppose $G$ has a RDRD-function with weight at most $k$. Among all such functions, let $g=(V_0,V_1,V_2,V_3)$ be a $\gamma_{rdR}$-function chosen satisfying the conditions.
\begin{itemize}
\item[(C1)] $|Y\cap V_3|$ is minimized.
\item[(C2)] Subject to Condition (C1): $|B \cap V_3|$ is minimized.
\end{itemize}
Denoted $V_0[B\cup X \setminus z]$ as the vertices labelled by 0 under $g$ in $B\cup X \setminus z$. Since $y_i$ is a pendent vertex and the definition, then $g(y_i) \neq 0$, for $i \in [3q]$. First, we claim that $g(z)=3$ and $g(v)=0$ for $v\in \{z_1,z_2,z_3,z_4\}$. We will prove the claim as follows. If $g(z)\neq 0$ or $g(z)=0$ and $G[V_0[B\cup X \setminus z]]$ has no isolated vertex, then $w(g(z,z_1,z_2,z_3,z_4)) \geq 6$, we define a new RDRD-function $g'$ as reassigning $z$ with $3$, $z_i=0$ for $i \in [4]$. And hence $w(g')<w(g)$ for $G$, a contradiction. Otherwise, there exists only one vertex in $B$ with labelled with 0, say $c'_j$. Now, we consider the weight under $g$ of the set $S=\{c_j,c'_j,x_i,y_i,z,z_1,z_2,z_3,z_4\}$ is at least 10, then we define a new RDRD-funciton $g'$ as reassigning $z=3$, $y_i=2$, $c_j=2$, the remaining vertices in $S$ with $0$. And hence $w(g')<w(g)$ for $G$, a contradiction. Note that every neighbor of $z$ can be defended by $z$, including vertices of $B$. Next, we want to prove that $g(x_i,y_i) \in \{2,3\}$ for every $i \in [3q]$. By the definition of RDRD-function, we know that $g(x_i,y_i) \geq 2$. Suppose $g(x_i,y_i) \geq 4$. If $g(x_i)=3$, then $g(y_i)\geq 1$, and reassigning $x_i$ by 2, a contradiction. If $g(x_i)=2$, then $g(y_i) \geq 2$, and reassigning $y_i$ by $1$, a contradiction. Finally, we claim that $g(v)\in \{0,2\}$ for $v \in B$, and further $g(c_j)=g(c'_j)=2$ is not exist for every $i \in [t]$. If there exists some vertex $v \in B$ satisfying $g(v)=1$, then reassigning $v$ by 0, a contradiction. Now suppose there exists a vertex $v\in B$ with $g(v)=3$ and the three neighbors in $X$ of $v$ is denoted by $x_1,x_2,x_3$. If $y_i=3$ for some $i \in [3]$, then reassigning $y_i=2$, contradiction to (C1). Otherwise, reassigning $v=2$, it is still a RDRD-function, contradiction to (c2). And further, if $g(c_j)=g(c'_j)=2$ for some $j \in [t]$. We redefine $g(c_i)=2$ and $g(c'_i)=0$, then we get a RDRD-function with a smaller weight, a contradiction.

Without loss of generality, we assume that $g(c'_j)=0$ for all $j \in [t]$. Now assume that the number of $g(x_i,y_i)=2$ is $a$, and the weight of each of the other $3q-a$ pairs is $3$. Let $b$ be the number of the vertices labelled by 2 in $B$. Then $$3(3q-a) + 2a + 2b + 3 \le 8q + 3,$$ which implies $a-2b\ge q$. Further, every vertex of $B$ has 3 neighbors in $X$, thus $a \leq 3b$. Note that the fact $a \leq 3q$, we obtain $a = 3q$ and $b = q$. Consequently, $C_S=\{C_j:g(c_j)=2\}$ is an exact cover of $C$.



\end{proof}
\section{Strong Product}
In this section, we first present some bounds on the restrained double Roman domination number of strong product graphs, moreover we determine the exact values of restrained double Roman domination numbers on $P_2\boxtimes P_m$ and $P_3\boxtimes P_m$. To begin with, we recall the following known results.
\begin{theorem}\label{str1}\cite{Imrich}
For any graphs $G$ and $H$, $$max\{P_2(G)\gamma(H),\gamma(G)P_2(H)\}\le\gamma{(G\boxtimes H)}\le 2\gamma(G)\gamma(H),$$ where $P_2(G)$ is a 2-packing of a graph $G$, which is a set of vertices in $G$ that are pair-wise at distance more that 2.
\end{theorem}

\begin{theorem}\label{str2}\cite{Mojdeha2021}
For any connected graph $G$ of order $n\ge3$, $\gamma_{rdR}(G)\le2n-2$.
\end{theorem}

\begin{observation}\label{ob1}
Let $G$ and $H$ be the connected graphs of order $n$ and $m$, respectively. Then $$2max\{P_2(G)\gamma(H),\gamma(G)P_2(H)\}\le\gamma_{rdR}{(G\boxtimes H)}\le 2nm-2.$$
\end{observation}
\begin{proof}
Recall the inequality $2\gamma(G)\le\gamma_{dR}(G)\le\gamma_{rdR}(G)$ for any graph $G$ and Theorem \ref{str1}, the lower bound can be obtained. Further note $V(G\boxtimes H)=|V(G)||V(H)|=nm$ and Theorem \ref{str2}, the upper bound is an immediate consequence.
\end{proof}

\begin{lemma}\label{str3}
For any connected graphs $G$ and $H$, let $f_1=(A_0,A_1,A_2,A_3)$ be a $\gamma_{rdR}(G)$-function and let $f_2=(B_0,B_1,B_2,B_3)$ be a $\gamma_{rdR}(H)$-function. Then
\begin{align}
\nonumber\gamma_{rdR}(G\boxtimes H)&\le\gamma_{rdR}(G)\gamma_{rdR}(H)-6|A_3||B_3|-3|A_3||B_2|-2|A_3||B_1|-3|A_2||B_3|-2|A_2||B_2|-\\
\nonumber &\hspace{7mm}|A_2||B_1|-|A_1||B_2|-2|A_1||B_3|.
\end{align}
\end{lemma}
\begin{proof}
We define the function $f$ on $G\boxtimes H$ as follows.
\begin{equation}{
f(u,v)=\left\{
\begin{array}{llll} \vspace{0.1cm}
3,                 &\hspace{2mm}(u,v)\in(A_3\times B_3)\cup(A_2\times B_3)\cup(A_3\times B_2),\\
2,                 &\hspace{2mm}(u,v)\in(A_2\times B_2),\\
1,                 &\hspace{2mm}(u,v)\in(A_1\times B_1)\cup(A_1\times B_2)\cup(A_1\times B_3)\cup(A_2\times B_1)\cup(A_3\times B_1),\\
0,                 &\hspace{2mm}otherwise.
\end{array}\notag
\right.}
\end{equation}
Note that the set $(A_0\times B_0)\cup(A_0\times B_1)\cup(A_1\times B_0)\cup(A_1\times B_1)$ can be restrained double Roman dominated by $(A_2\times B_2)\cup(A_2\times B_3)\cup(A_3\times B_2)\cup(A_3\times B_3)$, the set $(A_0\times B_2)\cup(A_1\times B_2)$ can be restrained double Roman dominated by $(A_2\times B_2)\cup(A_3\times B_2)$, the set $(A_0\times B_3)\cup(A_1\times B_3)$ can be restrained double Roman dominated by $(A_2\times B_3)\cup(A_3\times B_3)$, the set $(A_2\times B_0)\cup(A_2\times B_1)$ can be restrained double Roman dominated by $(A_2\times B_2)\cup(A_2\times B_3)$, and the set $(A_3\times B_0)\cup(A_3\times B_1)$ can be restrained double Roman dominated by $(A_3\times B_2)\cup(A_3\times B_3)$. Thus $f$ is a restrained double Roman domination function on $G\boxtimes H$.
Therefore, we have
\begin{align}
\nonumber \gamma_{rdR}(G\boxtimes H)&\le 3|A_3||B_3|+3|A_2||B_3|+3|A_3||B_2|+2|A_2||B_2|+|A_1||B_1|+|A_1||B_2|+\\
\nonumber &\hspace{7mm}|A_1||B_3|+|A_2||B_1|+|A_3||B_1|\\
\nonumber &=9|A_3||B_3|+6|A_3||B_2|+3|A_3||B_1|+6|A_2||B_3|+4|A_2||B_2|+2|A_2||B_1|+\\
\nonumber &\hspace{7mm}|A_1||B_1|+2|A_1||B_2|+3|A_1||B_3|-6|A_3||B_3|-3|A_3||B_2|-2|A_3||B_1|-\\
\nonumber &\hspace{5mm}3|A_2||B_3|-2|A_2||B_2|-|A_2||B_1|-|A_1||B_2|-2|A_1||B_3|\\
\nonumber &=3|A_3|(3|B_3|+2|B_2|+|B_1|)+2|A_2|(3|B_3|+2|B_2|+|B_1|)+|A_1|(3|B_3|+\\
\nonumber&\hspace{5mm}2|B_2|+|B_1|)-6|A_3||B_3|-3|A_3||B_2|-2|A_3||B_1|-3|A_2||B_3|-2|A_2||B_2|-\\
\nonumber&\hspace{5mm}|A_2||B_1|-|A_1||B_2|-2|A_1||B_3|\\
\nonumber &=(3|A_3|+2|A_2|+|A_1|)(3|B_3|+2|B_2|+|B_1|)-6|A_3||B_3|-3|A_3||B_2|-\\
\nonumber&\hspace{5mm}2|A_3||B_1|-3|A_2||B_3|-2|A_2||B_2|-|A_2||B_1|-|A_1||B_2|-2|A_1||B_3|\\
\nonumber&=\gamma_{rdR}(G)\gamma_{rdR}(H)-6|A_3||B_3|-3|A_3||B_2|-2|A_3||B_1|-3|A_2||B_3|-\\
\nonumber&\hspace{5mm}2|A_2||B_2|-|A_2||B_1|-|A_1||B_2|-2|A_1||B_3|.
\end{align}
\end{proof}

Now we give the following theorem by using the above lemma.

\begin{theorem}\label{str4}
For any connected graphs $G$ and $H$ of order $n\ge3$ and $m\ge3$ respectively, let $f_1=(A_0,A_1,A_2,A_3)$ be a $\gamma_{rdR}(G)$-function and let $f_2=(B_0,B_1,B_2,B_3)$ be a $\gamma_{rdR}(H)$-function. Then
\begin{align}
\nonumber\gamma_{rdR}(G\boxtimes H)&\le\gamma_{rdR}(G)\gamma_{rdR}(H)-6.
\end{align}
\end{theorem}
\begin{proof}
Note that the inequality in Lemma \ref{str3}, we will prove the result by considering the following cases.

\textbf{Case 1.} $|A_3|\ge1$ and $|B_3|\ge1$.

Obviously, the result holds by Lemma \ref{str3}.

\textbf{Case 2.} Either $|A_3|=0$ or $|B_3|=0$.

Without loose of generality, we may assume $|A_3|=0$ and $|B_3|\ge1$. Note that there must be $|A_2|\neq0$. If $|A_2|\ge2$, then $3|A_2||B_3|\ge6$. The result holds. If $|A_2|=1$, then $|A_1|\ge2$, since the order of $G$ is at least 3. It follows $3|A_2||B_3|+2|A_1||B_3|\ge7$. The result also holds. Similarly, the desired result can be obtained when $|B_3|=0$ and $|A_3|\ge1$.

\textbf{Case 3.} $|A_3|=|B_3|=0$.

We know that there must be $|A_2|\neq0$ and $|B_2|\neq0$. If $|A_2|\ge2$ and $|B_2|\ge2$, then $2|A_2||B_2|+|A_2||B_1|+|A_1||B_2|\ge8$. The result holds. If $|A_2|\ge2$ and $|B_2|=1$, then there must be $|B_1|\ge2$, since the order of $H$ is at least 3. It follows $2|A_2||B_2|+|A_2||B_1|+|A_1||B_2|\ge8$. Similarly, the desired result can be obtained if $|B_2|\ge2$ and $|A_2|=1$. Now we consider the only remaining case, that is, $|A_2|=|B_2|=1$, then there must be $|A_1|\ge2$ and $|B_1|\ge2$. It follows $2|A_2||B_2|+|A_2||B_1|+|A_1||B_2|\ge6$. The theorem holds.
\end{proof}

The upper bound in Theorem \ref{str4} can be achieved. For instance, if $G$ and $H$ are graphs of order $n$ and $m$, where $n,m\ge3$, containing a vertex of degree $n-1$ and $m-1$ respectively, and there is no pendant edge in $G$ and $H$. In such a case, $G\boxtimes H$ contains a vertex of degree $nm-1$ and there is no pendant edge in $G\boxtimes H$, since $d_{G\boxtimes H}(u,v)=d_{G}(u)d_{H}(v)+d_{G}(u)+d_{H}(v)$. Recall the result in \cite{Xi}, thus we have $\gamma_{rdR}(G\boxtimes H)=3=\gamma_{rdR}(G)\gamma_{rdR}(H)-6=3\times3-6$.

In the remaining part of this section, the exact values of $\gamma_{rdR}(P_2\boxtimes P_m)$ and $\gamma_{rdR}(P_3\boxtimes P_m)$ are determined.

In the sequel, let $f$ be a $\gamma_{rdR}(P_n\boxtimes P_m)$-function. We denote the set of vertices on the $j$-th column $\{v_{0,j},\ldots, v_{n-1,j}\}$, as $V^j$, that is $V^j=\{v_{i,j}:i\in\{0,\ldots,n-1\}\}$, and the weight of the $j$-th column $f_j=f(V^j)=\sum\nolimits_{v_{i,j}\in {V^j}}f(v_{i,j})$. Figure \ref{fig01} shows the graph of $P_n\boxtimes P_m$.
\begin{figure}[H]
\centering
\includegraphics[scale=0.4]{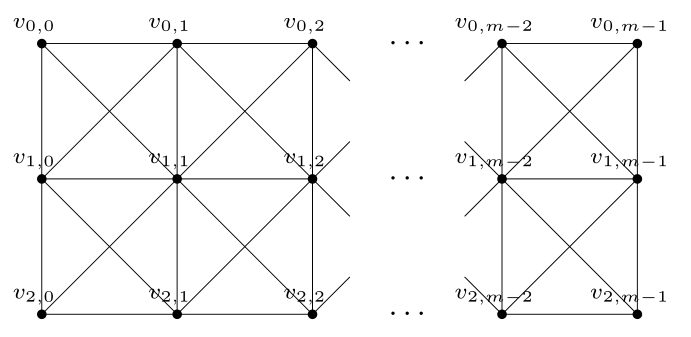}
\caption{Graph $P_n\boxtimes P_m$.}
\label{fig01}
\end{figure}
\begin{lemma}
Let $f$ be a $\gamma_{rdR}(P_n\boxtimes P_m)$-function, where $n=2~\text{or}~3$, the followings hold.
\begin{enumerate}[(1)]\label{stronglem}
  \item $f_{j-1}+f_j+f_{j+1}\ge3$ for $j\in\{1,\ldots,m-2\}$.
  \item $f_{0}+f_{1}\ge3$ and $f_{n-2}+f_{n-1}\ge3$.
\end{enumerate}
\end{lemma}
\begin{proof}
For $j\in\{0,\ldots,m-1\}$, if there exists a $j$ such that $0\le f_j\le2$, then there must be a vertex $v\in V^j$ assigned $0$ or $1$. By the definition of RDRD-function, $\sum_{u\in N[v]}f(u)\ge3$. Note that $N[v]\subseteq V^{j-1}\cup V^j\cup V^{j+1}$ for $j\in\{1,\ldots,m-2\}$, thus $f_{j-1}+f_j+f_{j+1}\ge3$. While if $v\in V^0$, then $N[v]\subseteq V^{0}\cup V^1$, thus $f_0+f_1\ge3$, Similarly we have $f_{n-2}+f_{n-1}\ge3$.
\end{proof}

\begin{theorem}\label{strongthm}
For $n=2~\text{or}~3$,
\begin{equation}{
\gamma_{rdR}(P_n\boxtimes P_m)=\left\{
\begin{array}{llll} \vspace{0.1cm}
m,                 &\hspace{2mm}m\equiv0(\bmod\ 3),\\
m+2,                 &\hspace{2mm}m\equiv1(\bmod\ 3),\\
m+1,                 &\hspace{2mm}m\equiv2(\bmod\ 3).
\end{array}\notag
\right.}
\end{equation}
\end{theorem}
\begin{proof}
Let $G=P_n\boxtimes P_m$, where $n=2~\text{or}~3$. We obtain the upper bounds of RDRD-numbers of $G$ through giving a RDRD-function on $G$ in the following cases.

\textbf{Case 1.} $m\equiv0(\bmod\ 3)$.

Define $f:V(G)\rightarrow\{0,1,2,3\}$ by $f(v_{1,3k+1})=3$ for $k=0,\ldots,\frac{m}{3}-1$ and $f(v_{i,j})=0$ otherwise. Clearly $f$ is a RDRD-function of $G$ with weight $\frac{m}{3}\times 3=m$, thus $\gamma_{rdR}(G)\le m$.

\textbf{Case 2.} $m\equiv1(\bmod\ 3)$.

Define $f:V(G)\rightarrow\{0,1,2,3\}$ by $f(v_{1,3k+1})=3$ for $k=0,\ldots,\lfloor\frac{m}{3}\rfloor-1$, $f(v_{1,m-2})=3$ and $f(v_{i,j})=0$ otherwise. Clearly $f$ is a RDRD-function of $G$ with weight $\frac{m-4}{3}\times 3+6=m+2$, thus $\gamma_{rdR}(G)\le m+2$.

\textbf{Case 3.} $m\equiv2(\bmod\ 3)$.

Define $f:V(G)\rightarrow\{0,1,2,3\}$ by $f(v_{1,3k+1})=3$ for $k=0,\ldots,\lfloor\frac{m}{3}\rfloor-1$, $f(v_{1,m-2})=3$ and $f(v_{i,j})=0$ otherwise. Clearly $f$ is a RDRD-function of $G$ with weight $\frac{m-2}{3}\times 3+3=m+1$, thus $\gamma_{rdR}(G)\le m+1$.

Now we prove the inverse inequality. For $n\equiv 0(\bmod\ 3)$, by Lemma \ref{stronglem} (1), we can obtain that $f(V)=\sum\nolimits_{j=0}^{m-1}f_j\ge \frac{m}{3}\times3=m$, that is $\gamma_{rdR}(G)\ge m$.

For $n\equiv 1(\bmod\ 3)$, by Lemma \ref{stronglem} (2), $f_{0}+f_{1}\ge3$ and $f_{m-2}+f_{m-1}\ge3$, we have
\begin{align}
\nonumber f(V)&=\sum\nolimits_{j=0}^{m-1}f_j=(f_{0}+f_{1}+f_{m-2}+f_{m-1})+\sum\nolimits_{j\in\{2,\ldots,m-3\}}f_{j}\\
\nonumber &\geq3+3+\frac{m-4}{3}\times3\\
\nonumber &=m+2.
\end{align}
Thus $\gamma_{rdR}(G)\ge m+2$.

For $n\equiv 2(\bmod\ 3)$, by Lemma \ref{stronglem} (2), $f_{0}+f_{1}\ge3$, we have
\begin{align}
\nonumber f(V)&=\sum\nolimits_{j=0}^{m-1}f_j=(f_{0}+f_{1})+\sum\nolimits_{j\in\{2,\ldots,m-1\}}f_{j}\\
\nonumber &\geq3+\frac{m-2}{3}\times3\\
\nonumber &=m+1.
\end{align}
Thus $\gamma_{rdR}(G)\ge m+1$.

From the above, the result holds.
\end{proof}

\section{Cardinal Product}
In this section, we study the restrained double Roman domination number of $G\times H$. For any two graphs, we present an upper bound and a lower bound of $\gamma_{rdR}(G\times H)$. Also we obtain the exact values of $\gamma_{rdR}(P_2\times G)$ for many types of graph $G$. Further, we determine the exact value of $\gamma_{rdR}(C_3\times C_m)$ by using a bagging approach.

We begin with the following known proposition.
\begin{proposition}\cite{Mojdeha2021}\label{pro1}

\begin{itemize}
\item[(1)] For a path $P_n$ $(n\ge4)$, $\gamma_{rdR}(P_n)=n+2.$
\item[(2)] For a cycle $C_n$ $(n\ge3)$,
\begin{equation}{
\gamma_{rdR}(C_{n})=\left\{
\begin{array}{llll} \vspace{0.1cm}
n,                 &\hspace{2mm}n\equiv0(\bmod\ 3),\\
n+2,                 &\hspace{2mm}otherwise.
\end{array}\notag
\right.}
\end{equation}
\end{itemize}
\end{proposition}

\begin{observation}
Let $G$ and $H$ be the graphs of order $n$ and $m$, respectively. Then
$$\gamma_{rdR}(G\times H)\ge\lceil\frac{3nm}{\Delta(G)\Delta(H)+1}\rceil,$$
In particular, if $G\times H$ is connected, then $$\gamma_{rdR}(G\times H)\le 2nm-2.$$
\end{observation}
\begin{proof}
In \cite{Volkmann2018}, Lutz Volkman proved that for any graph $G$ of order $n$ and $\Delta(G)\ge1$, $\gamma_{dR}(G)\ge\lceil\frac{3n}{\Delta(G)+1}\rceil.$ Since $\Delta(G\times H)=\Delta(G)\Delta(H)$ and $\gamma_{rdR}(G)\ge\gamma_{dR}(G)$, we get $$\gamma_{rdR}(G\times H)\ge\lceil\frac{3nm}{\Delta(G)\Delta(H)+1}\rceil.$$
Recall the upper bound of $\gamma_{rdR}(G)$ for any connected graph of order at least $3$ in Theorem \ref{str2}, $\gamma_{rdR}(G)\le2|V(G)|-2$. Further $V(G\times H)=|V(G)||V(H)|=nm$ for any two graphs. Thus we get $$\gamma_{rdR}(G\times H)\le2nm-2.$$
\end{proof}

\begin{theorem}
For any tree $T$ and any graph $G$ without odd cycle, we have
$$\gamma_{rdR}(P_2\times T)=2\gamma_{rdR}(T)<\gamma_{rdR}(P_2)\cdot \gamma_{rdR}(T),$$
$$\gamma_{rdR}(P_2\times G)=2\gamma_{rdR}(G)<\gamma_{rdR}(P_2)\cdot \gamma_{rdR}(G).$$
\end{theorem}
\begin{proof}
For cardinal product $G\times H$, if $G$ and $H$ are both bipartite, then $G\times H$ consists of two connected components. Further $P_2\times T$ and $P_2\times G$ consist of two disjoint copies of $T$ and $G$ respectively. Note that $\gamma_{rdR}(P_2)=3$, thus the theorem holds.
\end{proof}
By Proposition \ref{pro1}, we have Corollary \ref{coro1}.
\begin{corollary}\label{coro1}
\begin{equation}{
\gamma_{rdR}(P_2\times P_n)=\left\{
\begin{array}{llll} \vspace{0.1cm}
2n+4,                 & n\ge4,\\
2n+2,                 &n\le3.
\end{array}\notag
\right.}
\end{equation}
\end{corollary}

\begin{theorem}
For cardinal product of $P_2$ and $C_{2n+1}$, where $n\ge1$,
\begin{equation}{
\gamma_{rdR}(P_2\times C_{2n+1})=\left\{
\begin{array}{llll} \vspace{0.1cm}
4n+2,                 &n\equiv0(\bmod\ 6),\\
4n+4,                 &otherwise.
\end{array}\notag
\right.}
\end{equation}
\end{theorem}
\begin{proof}
For cardinal product $G\times H$, if $G$ and $H$ are not both bipartite, then $G\times H$ is connected. Further $P_2\times C_{2n+1}$ is connected and is a cycle $C_{4n+2}$. By Proposition \ref{pro1}, we can obtain that $\gamma_{rdR}(P_2\times C_{2n+1})=4n+2$ for $n\equiv0(\bmod\ 6)$, $\gamma_{rdR}(P_2\times C_{2n+1})=4n+4$ otherwise.
\end{proof}
In the remaining part of this section, we determine the exact value of $\gamma_{rdR}(C_3\times C_m)$ by using a bagging approach. Figure \ref{fig11}(a) shows the graph of $C_3\times C_m$. For more convenience and intuitiveness, we add virtual auxiliary rows and columns as shown in Figure \ref{fig11}(b).
\begin{figure}[H]
\centering
\includegraphics[scale=0.47]{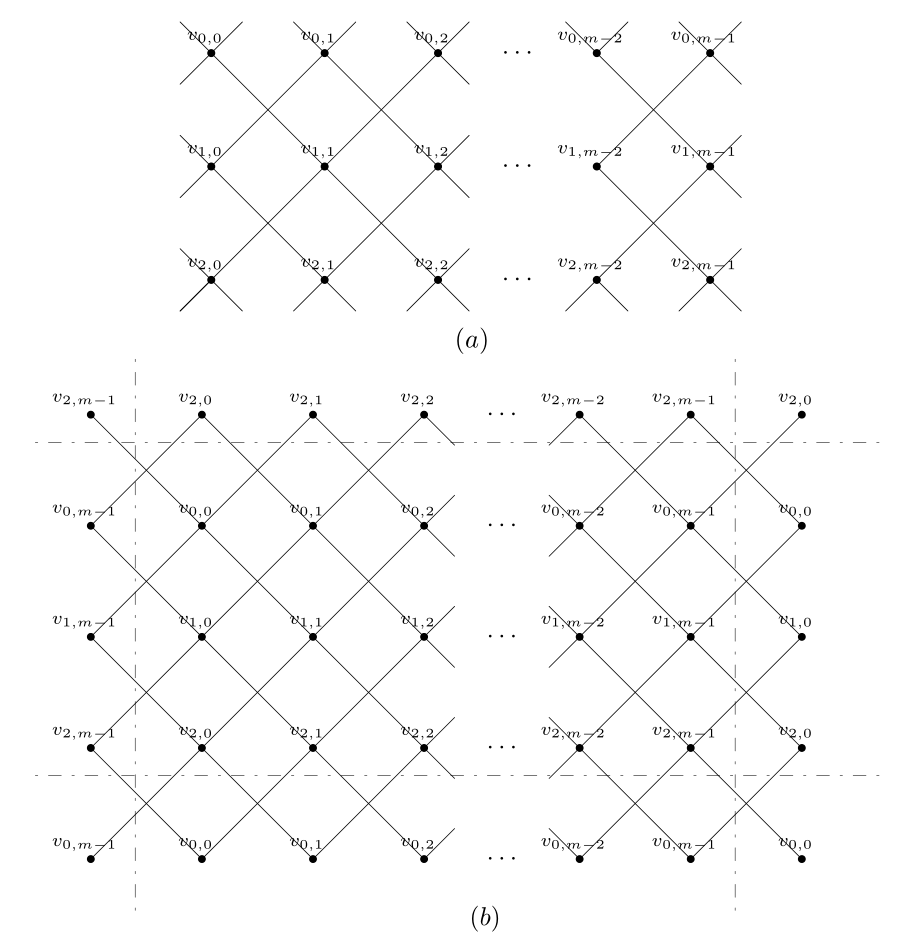}
\caption{Graph $C_3\times C_m$.}
\label{fig11}
\end{figure}
In the sequel, let $f$ be an arbitrary restrained double Roman domination function of $C_3\times C_m$. We denote the set of vertices on the $j$-th column, $\{v_{0,j}, v_{1,j}, v_{2,j}\}$, as $V^j$, that is $V^j=\{v_{i,j}:i\in\{0,1,2\}\}$, and the weight of the $j$-th column $f_j=f(V^j)=\sum\nolimits_{v_{i,j}\in {V^j}}f(v_{i,j})$.

\begin{lemma}\label{lembag}
For the graph $C_3\times C_m$, the following results hold.
\begin{enumerate}[(L1)]
  \item If $f_j=0$, then $f_{j-1}+f_{j+1}\ge6$.
  \item If $f_j=1$, then $f_{j-1}+f_{j+1}\ge5$.
  \item If $f_j=3$, then $f_{j-1}+f_{j+1}\ge3$.
  \item If $4\le f_j\le5$, then $f_{j-1}\ge2$ or $f_{j+1}\ge2$ holds.
\end{enumerate}
where $0\le j \le m-1$ and subscripts are taken modulo $m$.
\end{lemma}
\begin{proof}
\begin{enumerate}[(L1)]
  \item $f_{j}=0$ holds for any $0\le j\le m-1$. Without lose of generality, we may assume $f(v_{0,j-1})=0$, since $N(v_{1,j})\cap V_0\neq\emptyset$. By the definition of RDRD-function, we know that $f(v_{2,j-1})+f(v_{2,j+1})+f(v_{1,j-1})+f(v_{1,j+1})\ge3$, $f(v_{0,j+1})+f(v_{2,j-1})+f(v_{2,j+1})\ge3$ and $f(v_{1,j-1})+f(v_{1,j+1})+f(v_{0,j+1})\ge3$. Now we prove the result by considering the following cases.

\textbf{Case 1.1.} $f(v_{2,j-1})+f(v_{2,j+1})\ge3$.

If $f(v_{2,j-1})+f(v_{2,j+1})\ge3$, then $f_{j-1}+f_{j+1}\geq f(v_{1,j-1})+f(v_{1,j+1})+f(v_{0,j+1})+f(v_{2,j-1})+f(v_{2,j+1})\ge3+3=6$. The result holds.

\textbf{Case 1.2.} $f(v_{2,j-1})+f(v_{2,j+1})=2$.

If $f(v_{2,j-1})+f(v_{2,j+1})=2$, then consider $v_{0,j}$ and $v_{1,j}$ should be dominated, there are $f(v_{1,j-1})+f(v_{1,j+1})\ge2$ and $f(v_{0,j+1})\ge2$ respectively. It follows $f(v_{0,j+1})+f(v_{1,j-1})+f(v_{1,j+1})\ge2+2=4$. Thus $f(v_{0,j+1})+f(v_{1,j-1})+f(v_{1,j+1})+f(v_{2,j-1})+f(v_{2,j+1})\ge2+4=6.$ The result holds.

\textbf{Case 1.3.} $f(v_{2,j-1})+f(v_{2,j+1})\le1$.

If $f(v_{2,j-1})+f(v_{2,j+1})\le1$, consider $v_{0,j}$ and $v_{1,j}$ should be dominated, there are $f(v_{1,j-1})+f(v_{1,j+1})\ge3$ and $f(v_{0,j+1})=3$ respectively. It follows $f(v_{0,j+1})+f(v_{1,j-1})+f(v_{1,j+1})\ge3+3=6$. Thus the result also holds.

  \item $f_{j}=1$. If there exists a $j$ such that $f_j=1$, assume $f(v_{0,j})=1$. Note that $N(v_{1,j})\cup V_0\neq\emptyset$, without lose of generality, we may assume $f(v_{0,j-1})=0$ or $f(v_{2,j-1})=0$.

\textbf{Case 2.1.} $f(v_{0,j-1})=0$. By the definition of RDRD function, $f(v_{2,j-1})+f(v_{2,j+1})+f(v_{1,j-1})+f(v_{1,j+1})\ge2$, $f(v_{0,j+1})+f(v_{2,j-1})+f(v_{2,j+1})\ge3$ and $f(v_{1,j-1})+f(v_{1,j+1})+f(v_{0,j+1})\ge3$. Now we consider the following subcases for $f(v_{2,j-1})+f(v_{2,j+1})$.

\textbf{Case 2.1.1.} $f(v_{2,j-1})+f(v_{2,j+1})\ge2$. If $f(v_{2,j-1})+f(v_{2,j+1})\ge2$, then $f(v_{1,j-1})+f(v_{1,j+1})+f(v_{0,j+1})+f(v_{2,j-1})+f(v_{2,j+1})\ge3+2=5$. Thus the result holds.

\textbf{Case 2.1.2.} $f(v_{2,j-1})+f(v_{2,j+1})\le1$. If $f(v_{2,j-1})+f(v_{2,j+1})\le1$, then consider $v_{0,j}$ and $v_{1,j}$ should be dominated, there are $f(v_{1,j-1})+f(v_{1,j+1})\ge2$ and $f(v_{0,j+1})=3$ respectively. It follows $f(v_{0,j+1})+f(v_{1,j-1})+f(v_{1,j+1})\ge3+2=5$. Thus the result also holds.

\textbf{Case 2.2.} $f(v_{2,j-1})=0$. By the definition of RDRD function, $f(v_{2,j+1})+f(v_{1,j-1})+f(v_{1,j+1})\ge2$, $f(v_{0,j-1})+f(v_{0,j+1})+f(v_{2,j+1})\ge3$ and $f(v_{1,j-1})+f(v_{1,j+1})+f(v_{0,j-1})+f(v_{0,j+1})\ge3$. Now we consider the following subcases for $f(v_{1,j-1})+f(v_{1,j+1})$.

\textbf{Case 2.2.1.} $f(v_{1,j-1})+f(v_{1,j+1})\ge2$. If $f(v_{1,j-1})+f(v_{1,j+1})\ge2$, then $f(v_{1,j-1})+f(v_{1,j+1})+f(v_{0,j-1})+f(v_{0,j+1})+f(v_{2,j+1})\ge2+3=5$. Thus the result holds.

\textbf{Case 2.2.2.}  $f(v_{1,j-1})+f(v_{1,j+1})\le1$. If $f(v_{1,j-1})+f(v_{1,j+1})\le1$, then consider $v_{0,j}$ and $v_{2,j}$ should be dominated, there are $f(v_{2,j+1})=3$ and $f(v_{0,j-1})+f(v_{0,j+1})\ge3$ respectively. It follows $f(v_{2,j+1})+f(v_{0,j-1})+f(v_{0,j+1})\ge3+3=6$. Thus the result also holds.

  \item $f_j=3$. If there exists a $j$ such that $f_j=3$, then there must be at least a vertex $v\in V^j$ assigned $0$ or $1$. If there exists a vertex assigned $0$, then $f_{j-1}+f_{j+1}\ge3$. If there is no vertex assigned $1$ in $V^j$, then by the definition of RDRD function, $f(v_{2,j-1})+f(v_{2,j+1})+f(v_{1,j-1})+f(v_{1,j+1})\ge2$, $f(v_{0,j-1})+f(v_{0,j+1})+f(v_{2,j-1})+f(v_{2,j+1})\ge2$ and $f(v_{1,j-1})+f(v_{1,j+1})+f(v_{0,j-1})+f(v_{0,j+1})\ge2$. Sum the above inequalities, $2(f_{j-1}+f_{j+1})\ge6$, $f_{j-1}+f_{j+1}\ge3$. The result holds.

  \item $4\le f_j\le5$. If there exists a $j$ such that $4\le f_j\le5$, then there must be at least a vertex $v\in V^j$ assigned $0$ or $1$. By the definition of RDRD function, there are at least $f_{j-1}\ge2$ or $f_{j+1}\ge2$. The result holds.
\end{enumerate}
\end{proof}

\begin{theorem}
For $m\ge3$, $\gamma_{rdR}(C_3\times C_m)=2m.$
\end{theorem}
\makeatletter
\newenvironment{breakablealgorithm}
  {
   \begin{center}
     \refstepcounter{algorithm}
     \hrule height.8pt depth0pt \kern2pt
     \renewcommand{\caption}[2][\relax]{
       {\raggedright\textbf{\ALG@name~\thealgorithm} ##2\par}%
       \ifx\relax##1\relax 
         \addcontentsline{loa}{algorithm}{\protect\numberline{\thealgorithm}##2}%
       \else 
         \addcontentsline{loa}{algorithm}{\protect\numberline{\thealgorithm}##1}%
       \fi
       \kern2pt\hrule\kern2pt
     }
  }{
     \kern2pt\hrule\relax
   \end{center}
  }
\makeatother%
\begin{proof}
By L1 and L2 in Lemma \ref{lembag}, for each $f_j\le1$, there must be $f_{j-1}\ge3$ or $f_{j+1}\ge3$. Therefore we use the bagging approach to put $V^j$ $(0\le j\le m-1)$ into five $B_s$ bags. Actually, we collect vertices with the column weights at least $3$ and vertices with the column weights at most $1$ into one of the corresponding $B_s$ bags according to different cases. We can find that the average of the weights of all columns in the same bag is at least 2. In this process, if the vertices in a column $j$ have not been bagged, we denote $D[j]=0$, otherwise $D[j]=1$.
\renewcommand{\thealgorithm}{Procedure 1}
\begin{breakablealgorithm}
        \caption{Bagging Approach for $C_{3}\times C_{m}$}
        \begin{algorithmic}[1] 
            \STATE Initialization: $s=s_1=s_2=s_3=s_4=0$; $D[j]=0$ for $j=0$ up to $j=m-1$;
             \FOR {every $j$ with $f_{j}\geq 6\wedge D[j]=0$}
                    \STATE $s=s+1$; $D[j]=1$; $B_{s}=V^j$;
                    \IF{$f_{j-1}\le1$}
                    \STATE $B_{s}=B_{s}\cup V^{j-1}$;$D[j-1]=1$;
                    \ENDIF
                    \IF{$f_{j+1}\le1$}
                    \STATE $B_{s}=B_{s}\cup V^{j+1}$;$D[j+1]=1$;
                    \ENDIF
                    \STATE $\sum_{V^j\subseteq \cup_{t=1}^{s} B_{t}}f(j)\geq \frac{|\cup_{t=1}^{s} B_{t}|}{3}\times 2$; $s_1=s$;
             \ENDFOR

             \FOR {every $j$ with $4\le f_{j}\le 5\wedge D[j]=0$}
                    \STATE $s=s+1$; $D[j]=1$; $B_{s}=V^j$;
                    \IF{$f_{j-1}\le1$, by L4, $f_{j+1}\ge2$}
                    \STATE $B_{s}=B_{s}\cup V^{j-1}$;$D[j-1]=1$;
                    \ENDIF
                    \IF{$f_{j+1}\le1$, by L4, $f_{j-1}\ge2$}
                    \STATE $B_{s}=B_{s}\cup V^{j+1}$;$D[j+1]=1$;
                    \ENDIF
                    \STATE $\sum_{V^j\subseteq \cup_{t=s_{1}+1}^{s} B_{t}}f(j)\geq \frac{|\cup_{t={s_{1}+1}}^{s} B_{t}|}{3}\times 2$; $s_2=s$;
             \ENDFOR

              \FOR{every $j$ with $f_{j}=3\wedge D[j]=0$}
                    \STATE $s=s+1$; $D[j]=1$; $B_{s}=V^j$;
                    \IF{$f_{j-1}=1$, by L3, $f_{j+1}\ge2$}
                    \STATE $B_{s}=B_{s}\cup V^{j-1}$;$D[j-1]=1$;
                    \ENDIF
                    \IF{$f_{j+1}=1$, by L3, $f^{j-1}\ge2$}
                    \STATE $B_{s}=B_{s}\cup V^{j+1}$;$D[j+1]=1$;
                    \ENDIF
                    \STATE $\sum_{V^j\subseteq \cup_{t=s_{2}+1}^{s} B_{t}}f(j)\geq \frac{|\cup_{t={s_{2}+1}}^{s} B_{t}|}{3}\times 2$; $s_3=s$;
             \ENDFOR

             \FOR{every $j$ with $f_{j}=3\wedge D[j]=0$}
                    \STATE $s=s+1$; $D[j]=1$; $B_{s}=V^j$;
                    \IF{$f_{j-1}=0$, by L1 and L3, $f_{j-2}\ge3$, $f_{j+1}\ge3$}
                    \STATE $B_{s}=B_{s}\cup V^{j-1}\cup V^{j-2}$;$D[j-1]=D[j-2]=1$;
                    \ENDIF
                    \IF{$f_{j+1}=0$, by L1 and L3, $f_{j-1}\ge3$, $f_{j+2}\ge3$}
                    \STATE $B_{s}=B_{s}\cup V^{j+1}\cup V^{j+2}$;$D[j+1]=D[j+2]=1$;
                    \ENDIF
                    \STATE $\sum_{V^j\subseteq \cup_{t=s_{3}+1}^{s} B_{t}}f(j)\geq \frac{|\cup_{t={s_{3}+1}}^{s} B_{t}|}{3}\times 2$; $s_4=s$;
             \ENDFOR

                    \FOR{every $j$ with $f_{j}=2\wedge D[j]=0$}
                    \STATE $s=s+1$; $D[j]=1$; $B_{s}=V^j$;
                    \STATE $\sum_{V^j\subseteq \cup_{t=s_{4}+1}^{s} B_{t}}f(j)\geq \frac{|\cup_{t={s_{4}+1}}^{s} B_{t}|}{3}\times 2$;
             \ENDFOR
        \end{algorithmic}\label{procedure3}
\end{breakablealgorithm}

Thus we have
\begin{equation*}
\begin{array}{lll}
\vspace{0.1cm}
f(V)&=&\sum_{i=0}^{m-1}f(j)\\&=&\sum_{V^j\subseteq \cup_{t=1}^{s_1} B_{t}}f(j)+\sum_{V^j\subseteq \cup_{t=s_{1}+1}^{s_2} B_{t}}f(j)+\sum_{V^j\subseteq \cup_{t=s_{2}+1}^{s_3} B_{t}}f(j)\\
\vspace{0.1cm}&+&\sum_{V^j\subseteq \cup_{t=s_{3}+1}^{s_4} B_{t}}f(j)+\sum_{V^j\subseteq \cup_{t=s_{4}+1}^{s} B_{t}}f(j) \\
\vspace{0.1cm}&\geq & \frac{|\cup_{t=1}^{s_1} B_{t}|}{3}\times 2+\frac{|\cup_{t={s_{1}+1}}^{s_2} B_{t}|}{3}\times 2+\frac{|\cup_{t={s_{2}+1}}^{s_3} B_{t}|}{3}\times 2+\frac{|\cup_{t={s_{3}+1}}^{s_4} B_{t}|}{3}\times 2+\frac{|\cup_{t={s_{4}+1}}^{s} B_{t}|}{3}\times 2\\
&\geq &2m.
\end{array}
\end{equation*}

Define $f:V(C_3\times C_m)\rightarrow\{0,1,2,3\}$ by $f(v_{1,j})=2$ for $j\in\{0,\ldots,m-1\}$, $f(v_{i,j})=0$ otherwise. One can check that it is a RDRD-function on $C_3\times C_m$ with weight $2m$, thus $\gamma_{rdR}(C_3\times C_m)\le2m$. Therefore, $\gamma_{rdR}(C_3\times C_m)=2m$.
\end{proof}

\section{Corona Operator}
In this section, we study the restrained double Roman domination number of $G\odot H$ for any graph $G$ and $H$. We obtain the exact values of $\gamma_{rdR}(G\odot H)$ when $(H\ncong K_1)$, and present the sharp bounds of $\gamma_{rdR}(G\odot K_1)$. Further, the exact values of $\gamma_{rdR}(G\odot K_1)$ where $G$ is a complete graph, a cycle, a path or a complete bipartite graph are determined. In addition, we present the restrained double Roman domination number of $(G\odot K_1)\odot K_1$.

\begin{theorem}
For any connected graphs $G$ and any $H\ncong K_1$, then $\gamma_{rdR}(G\odot H)=3n$, where $n=|V(G)|$.
\end{theorem}

\begin{proof}
A RDRD-function of $G\odot H$ can be obtained by assigning $3$ to all vertices of $G$, and $0$ to the remaining vertices in $G\odot H$. Thus $\gamma_{rdR}(G\odot H)\le 3n$. Conversely, the $n$ copies of $H$ are pairwise disjoint in $G\odot H$. And further, each copy need a weight of at least 3, since the definition of a RDRD-function and $H \ncong K_1$. It follows that $\gamma_{rdR}(G\odot H)\ge 3n$. Thus the theorem holds.
\end{proof}

\begin{theorem}\label{GK1}
For any connected graph $G$ of order $n$, then $2n+1\le\gamma_{rdR}(G\odot K_1)\le3n$.
\end{theorem}
\begin{proof}
For $G\ncong K_1$, a RDRD-function can be obtained by assigning 3 to the pendant vertices of $G\odot K_1$, and $0$ to the remaining vertices. Otherwise, we assign 1 and 2 to the two vertices respectively. Thus $\gamma_{rdR}(G\odot K_1)\le 3n$, where $n$ is the order of $G$. Now we prove the inverse inequality. Each pendant vertex of $G\odot K_1$ belongs to either $V_2 \cup V_3$ or $V_1$ adjacent to the vertex in $V_2\cup V_3$. We claim that if all pendant vertices are assigned $2$, then at least one vertex in $G$ is assigned the value not less than 1. Otherwise, all vertices assigned $0$ in $G$ cannot be restrained double Roman dominated, which contradicts the definition of the RDRD-function. Thus $\gamma_{rdR}(G\odot H)\ge 2n+1$.
\end{proof}
Thereafter we show the bounds in Theorem \ref{GK1} are sharp.
\begin{theorem}
\begin{equation}{
\gamma_{rdR}(K_n\odot K_1)=\left\{
\begin{array}{llll} \vspace{0.1cm}
2n+1,                 &if \hspace{2mm}n\neq2,\\
6,                 &if \hspace{2mm}n=2.
\end{array}\notag
\right.}
\end{equation}
\end{theorem}

\begin{proof}
It is easy to see $\gamma_{rdR}(K_2\odot K_1)=6$. For $K_n\odot K_1$ $(n\neq2)$, we assign 1 to a pendant vertex, 2 to its neighbor and the remaining pendant vertices, and 0 otherwise. It is a RDRD-function of $K_n\odot K_1$ $(n\neq2)$ and thus $\gamma_{rdR}(K_n\odot K_1)\le2n+1.$ By Proposition \ref{GK1}, we can obtain $\gamma_{rdR}(K_n\odot K_1)=2n+1$ when $n\neq2$.
\end{proof}

The exact value of double Roman domination number on $C_n\odot K_1$ is obtained by Anu V. and Aparna Lakshmanan S. \cite{Anu2019}. In the sequel, we determine the exact values of restrained double Roman domination number on $C_n\odot K_1$ and $P_n\odot K_1$.
\begin{theorem}\cite{Anu2019}\label{th2.3}
\begin{equation}{
\gamma_{dR}(C_n\odot K_1)=\left\{
\begin{array}{llll}  \vspace{0.1cm}
\frac{7n}{3},                 &if \hspace{2mm}n=3k,\\ \vspace{0.1cm}
\frac{7n+2}{3},                 &if \hspace{2mm}n=3k+1,\\
\frac{7n+1}{3},                 &if \hspace{2mm}n=3k+2,\\
\end{array}\notag
\right.}
\end{equation}
That is $\gamma_{dR}(C_n\odot K_1)=\lceil\frac{7n}{3}\rceil$.
\end{theorem}

In the sequel, for $G\in\{C_n,P_n\}$, let $u_i\in V(G)$, ${u_i}'$ be the leaf neighbors of $u_i$ in $G\odot K_1$, where $i\in\{1,\ldots,n\}$. Let $f$ be an arbitrary RDRD-function on $G\odot K_1$, we denote $f_j=f(u_{j})+f({u'_j})$ $(1\le j\le n)$ and $f(V)=\sum_{j=1}^n(f_j)$.
\begin{lemma}\label{lem1}
For an arbitrary RDRD-function $f$ on $G \odot K_1$, we have
\begin{enumerate}[(1)]
  \item If $G=C_n$, then $f_j+f_{j+1}+f_{j+2}\ge7$, where $1\le j\le n$ and subscripts are taken modulo $n$.
  \item If $G=P_n$, then $f_j+f_{j+1}+f_{j+2}\ge7$, where $1\le j\le n-2$.
\end{enumerate}
\end{lemma}
\begin{proof}
By the definition of RDRD-function, we can easily get $f_j\ge2$$(1\le j\le n)$, so $f_j+f_{j+1}+f_{j+2}\ge6$. Suppose $f_j+f_{j+1}+f_{j+2}=6$, it means that $f_j=f_{j+1}=f_{j+2}=2$. The only possible case is $f(u_{j})=f(u_{j+1})=f(u_{j+2})=0$ and $f({u'_j})=f({u'_{j+1}})=f({u'_{j+2}})=2$. In this case, $u_{j+1}$ cannot be dominated, a contradiction. Thus $f_j+f_{j+1}+f_{j+2}\ge7$. Furthermore, the results in Lemma \ref{lem1} (1) and (2) are true for $C_n\odot K_1$ and $P_n\odot K_1$ respectively.
\end{proof}

\begin{theorem}\label{procnk1}
\begin{equation}{
\gamma_{rdR}(C_n\odot K_1)=\left\{
\begin{array}{llll}  \vspace{0.1cm}
\lceil\frac{7n}{3}\rceil,                 &if \hspace{2mm}n\equiv0,1(\bmod 3),\\
\lceil\frac{7n}{3}\rceil+1,                 &if \hspace{2mm}n\equiv2(\bmod 3).
\end{array}\notag
\right.}
\end{equation}
\end{theorem}
\begin{proof}
Figure \ref{fig1} shows the graph of $C_n\odot K_1$.
\begin{figure}[H]
\centering
\includegraphics[scale=0.45]{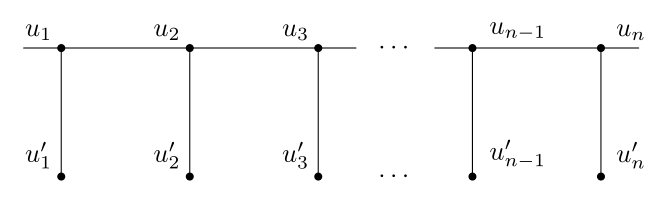}
\caption{Graph $C_n\odot K_1$.}
\label{fig1}
\end{figure}
\textbf{Case 1.} $n\equiv 0(\bmod\ 3)$. Define $f:V(C_n\odot K_1)\rightarrow\{0,1,2,3\}$ by $f(u_{3k+3})=f(u'_{3k+1})=f(u'_{3k+2})=2$, $f(u'_{3k+3})=1$ for $k=0,\ldots,\frac{n}{3}-1$ and $f(u_{i})=f(u'_{i})=0$ otherwise. It is clear that $f$ is a RDRD-function of $C_n\odot K_1$ with weight $\frac{7n}{3}$, thus $\gamma_{rdR}(C_n\odot K_1)\le\frac{7n}{3}$.

\textbf{Case 2.} $n\equiv 1(\bmod\ 3)$. Define $f:V(C_n\odot K_1)\rightarrow\{0,1,2,3\}$ by $f(u_{3k+3})=f(u'_{3k+1})=f(u'_{3k+2})=f(u_{n})=2$, $f(u'_{3k+3})=f(u'_{n})=1$ for $k=0,\ldots,\lfloor\frac{n}{3}\rfloor-1$ and $f(u_{i})=f(u'_{i})=0$ otherwise. It is clear that $f$ is a RDRD-function of $C_n\odot K_1$ with $f(V)=\frac{n-1}{3}\times7+3=\frac{7n+2}{3}=\lceil\frac{7n}{3}\rceil$, thus $\gamma_{rdR}(C_n\odot K_1)\le\lceil\frac{7n}{3}\rceil$.

\textbf{Case 3.} $n\equiv 2(\bmod\ 3)$. Define $f:V(C_n\odot K_1)\rightarrow\{0,1,2,3\}$ by $f(u_{3k+3})=f(u'_{3k+1})=f(u'_{3k+2})=f(u_{n})=f(u'_{n-1})=2$, $f(u'_{3k+3})=f(u_{n-1})=f(u'_{n})=1$ for $k=0,\ldots,\lfloor\frac{n}{3}\rfloor-1$ and $f(u_{i})=f(u'_{i})=0$ otherwise. It is clear that $f$ is a RDRD-function of $C_n\odot K_1$ with $f(V)=\frac{n-2}{3}\times7+6=\frac{7n+4}{3}=\lceil\frac{7n}{3}\rceil+1$, thus $\gamma_{rdR}(C_n\odot K_1)\le\lceil\frac{7n}{3}\rceil+1$.

Now we prove the inverse inequality. Since $\gamma_{rdR}(G)\ge\gamma_{dR}(G)$ for any graph $G$, and the known result in Theorem \ref{th2.3}, we can obtain that $\gamma_{rdR}(C_n\odot K_1)\ge\lceil\frac{7n}{3}\rceil$. Thus $\gamma_{rdR}(C_n\odot K_1)=\lceil\frac{7n}{3}\rceil$ for $n\equiv0,1(\bmod\ 3)$ can be obtained. We claim that $\gamma_{rdR}(C_n\odot K_1)\ge\lceil\frac{7n}{3}\rceil+1$ for $n\equiv2(\bmod\ 3)$. Let $f$ be a $\gamma_{rdR}(C_n\odot K_1)$-function.

\begin{claim}\label{claimge9}
For $C_n\odot K_1$, if there exists a $j$ such that $f_{j}+f_{j+1}+f_{j+2}\ge9$, then $\gamma_{rdR}(C_n\odot K_1)\ge\lceil\frac{7n}{3}\rceil+1$ for $n\equiv2(\bmod\ 3)$.
\end{claim}
\begin{proof}
If there exists a $j_1$ such that $f_{j_1}+f_{j_{1}+1}+f_{j_{1}+2}\ge9$, then
\begin{align}
\nonumber 3 f(V)&=3\sum\nolimits_{j=1}^{n}f_j=(f_{j_1}+f_{j_1+1}+f_{j_1+2})+\sum\nolimits_{j\in\{1,\ldots,n\}-\{j_1\}}(f_j+f_{j+1}+f_{j+2})\\
\nonumber &\geq9+7\times (n-1)\\
\nonumber &=7n+2.
\end{align}
Thus $\gamma_{rdR}(C_n\odot K_1)\ge\frac{7n+2}{3}=\lceil\frac{7n}{3}\rceil+1$ for $n\equiv2(\bmod\ 3)$.
\end{proof}
By Claim \ref{claimge9}, we may assume that $f_{j}+f_{j+1}+f_{j+2}\le8$ for all $j$$(1\le j\le n)$, where subscripts are taken modulo $n$, for otherwise the desired result follows. Recall Lemma \ref{lem1}, $7\le f_{j}+f_{j+1}+f_{j+2}\le8$ for all $j$.
\begin{claim}\label{claimeq8}
For $C_n\odot K_1$, let $S=\{j\colon f_{j}+f_{j+1}+f_{j+2}=8\}$, there must be $|S|\ge2$ for $n\equiv2(\bmod\ 3)$.
\end{claim}
\begin{proof}
 Suppose $|S|=0$, it means that $f_j+f_{j+1}+f_{j+2}=7$ for all $j$. In this case, $(f_j,f_{j+1},f_{j+2})\in\{(2,2,3),(2,3,2),(3,2,2)\}$. Without loss of generality, assume $f_1=3$. It follows $f_2=f_3=2$, $f_4=3$. Continue in this way, we have $f(u_{3k+1})=f(u'_{3k+2})=2$, $f(u'_{3k+1})=1$ for $k=0,\ldots,\lfloor\frac{n}{3}\rfloor$, $f(u'_{3k+3})=2$ for $k=0,\ldots,\lfloor\frac{n}{3}\rfloor-1$ and $f(u_{i})=f(u'_{i})=0$ otherwise. Note that $f(u_{n})=0$ and $N(u_{n})\cap V_0=\emptyset$, which contradicts the definition of RDRD-function. Suppose $|S|=1$, say $j_1\in S$. In this case, for any $\gamma_{rdR}$-function, $(f_{j_1},f_{{j_1}+1},f_{{j_1}+2})\in\{(2,3,3),(3,3,2),(3,2,3)\}$. If $(f_{j_1},f_{{j_1}+1},f_{{j_1}+2})\in\{(2,3,3),(3,3,2)\}$, then $({j_1}+1)\in S$ or $({j_1}-1)\in S$, that is there must be more $j'$ belongs to $S$. This is a contradiction to $|S|=1$. If $(f_{j_1},f_{{j_1}+1},f_{{j_1}+2})=(3,2,3)$, then $f(u_{j_1+1})=0$ and $f(u'_{j_1+1})=2$. There must be a neighbor of $u_{j_1+1}$ assigned 0 by the RDRD function, say $f(u_{j_1+2})=0$. It follows $f(u'_{j_1+2})=3$. Since $|S|=1$, $f_{j_1+3}=f_{j_1+4}=2$. Note that $f(u_{{j_1}+3})=0$ and $u_{{j_1+3}}$ cannot be dominated. This is a contradiction.
\end{proof}
Thus for $n\equiv2(\bmod\ 3)$, there are at least two columns $j_1,j_2$ belong to the $S$ defined in Claim \ref{claimeq8}.
\begin{align}
\nonumber 3 f(V)&=3\sum\nolimits_{j=1}^{n}f_j=\sum\nolimits_{j\in\{j_1,j_2\}}(f_j+f_{j+1}+f_{j+2})+\sum\nolimits_{j\in\{1,\ldots,n\}-\{j_1,j_2\}}(f_j+f_{j+1}+f_{j+2})\\
\nonumber &\geq8+8+7\times (n-2)\\
\nonumber &=7n+2.
\end{align}
Thus $\gamma_{rdR}(C_n\odot K_1)\ge\frac{7n+2}{3}=\lceil\frac{7n}{3}\rceil+1$ for $n\equiv2(\bmod\ 3)$.
\end{proof}

\begin{theorem}\label{propnk1}
\begin{equation}{
\gamma_{rdR}(P_n\odot K_1)=\left\{
\begin{array}{llll} \vspace{0.1cm}
\lceil\frac{7n}{3}\rceil,                 &if \hspace{2mm}n\equiv1(\bmod 3),\\
\lceil\frac{7n}{3}\rceil+1,                 &if \hspace{2mm}n\equiv0, 2(\bmod 3).
\end{array}\notag
\right.}
\end{equation}
\end{theorem}
\begin{proof}
Figure \ref{fig2} shows the graph of $P_n\odot K_1$.
\begin{figure}[H]
\centering
\includegraphics[scale=0.45]{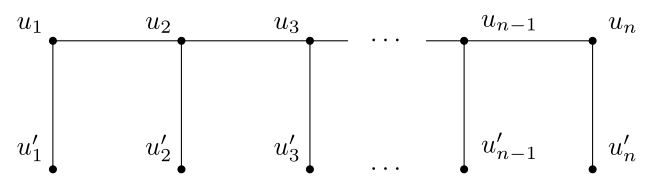}
\caption{Graph $P_n\odot K_1$.}
\label{fig2}
\end{figure}
\textbf{Case 1.} $n\equiv 0(\bmod\ 3)$. Define $f:V(P_n\odot K_1)\rightarrow\{0,1,2,3\}$ by $f(u_{3k+1})=f(u'_{3k+2})=2$, $f(u'_{3k+1})=1$ for $k=0,\ldots,\frac{n}{3}-1$, $f(u'_{3k+3})=2$ for $k=0,\ldots,\frac{n}{3}-2$, $f(u'_{n})=3$ and $f(u_{i})=f(u'_{i})=0$ otherwise. It is clear that $f$ is a RDRD-function of $P_n\odot K_1$ with $f(V)=\frac{n-3}{3}\times7+8=\frac{7n+3}{3}=\frac{7n}{3}+1$, thus $\gamma_{rdR}(P_n\odot K_1)\le\frac{7n}{3}+1$.

\textbf{Case 2.} $n\equiv 1(\bmod\ 3)$. Define $f:V(P_n\odot K_1)\rightarrow\{0,1,2,3\}$ by $f(u_{3k+1})=2$, $f(u'_{3k+1})=1$ for $k=0,\ldots,\lfloor\frac{n}{3}\rfloor$, $f(u'_{3k+2})=f(u'_{3k+3})=2$ for $k=0,\ldots,\lfloor\frac{n}{3}\rfloor-1$ and $f(u_{i})=f(u'_{i})=0$ otherwise. It is clear that $f$ is a RDRD-function of $P_n\odot K_1$ with $f(V)=\frac{n-1}{3}\times7+3=\frac{7n+2}{3}=\lceil\frac{7n}{3}\rceil$, thus $\gamma_{rdR}(P_n\odot K_1)\le\lceil\frac{7n}{3}\rceil$.

\textbf{Case 3.} $n\equiv 2(\bmod\ 3)$. Define $f:V(P_n\odot K_1)\rightarrow\{0,1,2,3\}$ by $f(u_{3k+1})=f(u'_{3k+2})=2$, $f(u'_{3k+1})=1$ for $k=0,\ldots,\lfloor\frac{n}{3}\rfloor$, $f(u'_{3k+3})=2$ for $k=0,\ldots,\lfloor\frac{n}{3}\rfloor-1$, $f(u_{n})=1$ and $f(u_{i})=f(u'_{i})=0$ otherwise. It is clear that $f$ is a RDRD-function of $P_n\odot K_1$ with $f(V)=\frac{n-2}{3}\times7+6=\frac{7n+4}{3}=\lceil\frac{7n}{3}\rceil+1$, thus $\gamma_{rdR}(P_n\odot K_1)\le\lceil\frac{7n}{3}\rceil+1$.

Now we prove the inverse inequality. Since $\gamma_{rdR}(P_n\odot K_1)\ge\gamma_{rdR}(C_n\odot K_1)$ and Theorem \ref{procnk1}, we can obtain that $\gamma_{rdR}(P_n\odot K_1)\ge\lceil\frac{7n}{3}\rceil$ for $n\equiv0,1(\bmod\ 3)$ and $\gamma_{rdR}(P_n\odot K_1)\ge\lceil\frac{7n}{3}\rceil+1$ for $n\equiv2(\bmod\ 3)$. Thus the results in Theorem \ref{propnk1} are true for $n\equiv1,2(\bmod\ 3)$. Now we claim that $\gamma_{rdR}(P_n\odot K_1)\ge\frac{7n}{3}+1$ for $n\equiv0(\bmod\ 3)$. Let $f$ be a $\gamma_{rdR}(P_n\odot K_1)$-function. By Lemma \ref{lem1}, for $P_n\odot K_1$, $f_j+f_{j+1}+f_{j+2}\ge7$, where $1\le j\le n-2$. In particular, $f_{3k+1}+f_{3k+2}+f_{3k+3}\ge7$, where $0\le k\le \frac{n}{3}-1$. Suppose $\gamma_{rdR}(P_n\odot K_1)=\frac{7n}{3}$, it means that $f_{3k+1}+f_{3k+2}+f_{3k+3}=7$ for $0\le k\le \frac{n}{3}-1$. Note that $f_1\ge3$, the only possible case is $f_{3k+1}=3, f_{3k+2}=f_{3k+3}=2$$(0\le k\le \frac{n}{3}-1)$. Note that $f_n=2$ and $f(u_{n})=0$, $u_{n}$ cannot be restrained double Roman dominated, which  contradicts the definition of RDRD-function. Thus $\gamma_{rdR}(P_n\odot K_1)\ge\frac{7n}{3}+1$ for $n\equiv0(\bmod\ 3)$.
\end{proof}

The following are some results of double Roman domination number on the corona operation presented in \cite{Anu2019}. Here we study the impact of corona operation on restrained double Roman domination number.

\begin{theorem}\label{Anu}\cite{Anu2019}
\begin{equation}{
\gamma_{dR}(K_{p,q}\odot K_1)=\left\{
\begin{array}{llll}
2(p+q)+1,                 &if \hspace{4.4mm}p=1 \ or\ q=1,\\
2(p+q+1),                 &otherwise.
\end{array}\notag
\right.}
\end{equation}
\end{theorem}

\begin{theorem}\label{Anu2}\cite{Anu2019}
For any graph $G$, $\gamma_{dR}((G\odot K_1)\odot K_1)=5n$, where $n=|V(G)|$.
\end{theorem}

\begin{proposition}\label{kpqk1}
\begin{equation}{
\gamma_{rdR}(K_{p,q}\odot K_1)=\left\{
\begin{array}{llll}
3(p+q),                 &if \hspace{4.4mm}p=1 \ or\ q=1,\\
2(p+q+1),                 &otherwise.
\end{array}\notag
\right.}
\end{equation}
\end{proposition}
\begin{proof}
Let $\{u_1,\ldots,u_p,v_1,\ldots,v_q\}$ be the vertex of $V(K_{p,q})$. Let ${u'_i}$ and ${v'_j}$ be the leaf neighbors of $u_i$ and $v_j$ respectively in $K_{p,q}\odot K_1$, where $i\in\{1,\ldots,p\}$ and $j\in\{1,\ldots,q\}$.

\textbf{Case 1}. $p=1$ or $q=1$. Without loss of generality, assume $p=1$. Note that $K_{1,q}\odot K_1$ is a wounded spider $ws(1,q,q-1)$ obtained from subdividing $q-1$ edges of a star $K_{1,q}$. We refer to the result concerning $\gamma_{rdR}(T)$ in \cite{Mojdeha2021}, it implying that if $T$ is a wounded spider $ws(1,q,q-1)$ with order $n$, then $\gamma_{rdR}(ws(1,q,q-1))=\lceil\frac{3n-1}{2}\rceil$. That is $\gamma_{rdR}(ws(1,q,q-1))=3(p+q)$ since $n=2(p+q)$ in $K_{1,q}\odot K_1$.

\textbf{Case 2.} $p,q\ge2$. We define a function RDRD function $f$ with weight $2(p+q+1)$ as follows.
\begin{equation}{
f(u)=\left\{
\begin{array}{llll}
2,                 &if\hspace{1.6mm} u=u_1\ or\ u=v_1\ or\ u={u'_i}(i\neq1)\ or\ u={u'_j}(j\neq1),\\
1,                 &if\hspace{1.6mm} u={u'_1} \ or\ u={v'_1},\\
0,                 &otherwise.
\end{array}\notag
\right.}
\end{equation}
Thus $\gamma_{rdR}(K_{p,q}\odot K_1)\le 2(p+q+1)$.

Recall the result in Theorem \ref{Anu}, that is $\gamma_{dR}(K_{p,q}\odot K_1)=2(p+q+1)$ for $p,q\ge2$. Thus $\gamma_{rdR}(K_{p,q}\odot K_1)\ge\gamma_{dR}(K_{p,q}\odot K_1)=2(p+q+1)$.

The result is true.
\end{proof}

Combining the above theorems, we present the bounds in Theorem \ref{GK1} are sharp. If $G \in\{K_n(n\neq2), C_3, K_1\}$, then the lower bound $2n+1$ of $\gamma_{rdR}(G\odot K_1)$ is arrived, where $n=|V(G)|$; if $G\in\{K_{p,q}(p=1\ or\ q=1),K_1, K_2\}$, the upper bound $3n$ of $\gamma_{rdR}(G\odot K_1)$ is arrived, where $n=|V(G)|$.

\begin{theorem}
For any graph $G$, $\gamma_{rdR}((G\odot K_1)\odot K_1)=5n$, where $n=|V(G)|$.
\end{theorem}
\begin{proof}
Let $u_i\in V(G)$ for $i\in\{1,\ldots,n\}$, $v_i$ be the neighbor of $u_i$ in $G\odot K_1$. Let ${u'_i}$ and ${v'_i}$ be the leaf neighbors of $u_i$ and $v_i$ respectively in $(G\odot K_1)\odot K_1$. Now we define a function RDRD function $f$ with weight $5n$ as follows, where $n=|V(G)|$.
\begin{equation}{
f(u)=\left\{
\begin{array}{llll}
2,                 &if\hspace{1.6mm} u={u'_i}\ or\ u={v_i}, \hspace{1.4mm}where\hspace{1.4mm} i\in\{1,\ldots,n\},\\
1,                 &if\hspace{1.6mm} u={v'_i}  \hspace{1.4mm}where\hspace{1.4mm} i\in\{1,\ldots,n\},\\
0,                 &otherwise.
\end{array}\notag
\right.}
\end{equation}
Thus $\gamma_{rdR}((G\odot K_1)\odot K_1)\le 5n$.

Recall the result in Theorem \ref{Anu}, that is $\gamma_{dR}((G\odot K_1)\odot K_1)=5n$, where $n=|V(G)|$. Thus $\gamma_{rdR}((G\odot K_1)\odot K_1)\ge\gamma_{dR}((G\odot K_1)\odot K_1)=5n$.
\end{proof}



\begin{thebibliography}{}
\bibitem{ACS}
H. A. Ahangar, M. Chellali and S. M. Sheikholeslami, On the double Roman domination in graphs, Discrte Appl. Math., 232(2017), 1-7.

\bibitem{Beeler}
R. A. Beeler, T. W. Haynes and S. T. Hedetniemi, Double Roman domination, Discrete Appl. Math., 211(2016), 23-29.

\bibitem{Chen11}
X. Chen, J. Liu and J. X. Meng, Total restrained domination in graphs, Comput. Math. Appl., 62(2011), 2892-2898.

\bibitem{cardinalpp3}
R. Ch\'{e}rifi, S. Gravier, X. Lagraula, C. Payan and I. Zighem, Domination number of the cross product of paths, Discrete Appl. Math., 94(1999), 101-139.

\bibitem{Domke}
G. S. Domke, J. H. Hattingh, S. T. Hedetniemi, R. C. Laskar and L. R. Markus, Restrained domination in graphs, Discrete Math., 203,1-3(1999), 61-69.

\bibitem{Imrich}
W. Imrich and S. Klav\v{z}ar. Product Graphs: structure and recognition. Wiley-Interscience, NewYork, USA, 2000.

\bibitem{cardinal1}
P. K. Jha, Perfect $r$-domination in the Kronecker product of two cycles, with an application to diagonal/toroidal mesh, Inform. Process. Lett., 87(3)(2003), 163-168.

\bibitem{cardinal4}
S. Klav\v{z}ar and B. Zmazek, On a Vizing-like conjecture for direct product graphs, Discrete Math., 156(1996), 243-246.

\bibitem{cardinalpp1}
A. Klobu\v{c}ar, Domination numbers of cardinal products, Math. Slovaca., 49(4)(1999), 387-402.

\bibitem{cardinalpp2}
A. Klobu\v{c}ar, Domination numbers of cardinal products $P_6\times P_n$, Math. Commun., 4(1999), 241-250.

\bibitem{cardinal2}
A. Klobu\v{c}ar and A. Klobu\v{c}ar, Properties of double Roman domination on cardinal products of graphs, Ars Math. Contemp., 9(2)(2020), 337-349.

\bibitem{cardinal3}
A. C. Mart\'{i}nez, D. Kuziak, I. Peterin and I. G. Yero, Dominating the direct product of two graphs through total Roman strategies, Mathematics, 8(9)2020, 1438.

\bibitem{Mojdeha2021}
D. A. Mojdeha, I. Masoumib and L. Volkmann, Restrained double Roman domination of a graph, [arXiv:2106.08501].

\bibitem{RoushiniLeely1}
P. R. L. Pushpam and  C. Suseendran, Secure restrained domination in Graphs, Math. Comput. Sci., 9(2015), 239-247.

\bibitem{RoushiniLeely2}
P. R. L. Pushpam and S. Padmapriea, Restrained Roman domination in graphs, Transactions on Combinatorics, 4(1)(2015), 1-17.

\bibitem{Krzywkowski}
N. J. Rad and M. Krzywkowski, On the restrained Roman domination in graphs, In Colourings, Independence and Domination-CID 2015, University of Zielona Gora.

\bibitem{Samadi}
B. Samadi, M. Alishahi, I. Masoumi and D. A. Mojdeh, Restrained Italian domination in graphs, RAIRO-Oper. Res., 55(2021), 319-332.

\bibitem{Samadi1}
B. Samadi, N. Soltankhah and D. A. Mojdeh, Restrained condition on double Roman dominating functions, arXiv:2109.066666v1.


\bibitem{Anu2019}
V. Anu and S. A. Lakshmanan, Impact of some graph operations on double Roman domination number, arXiv:1908.06859.

\bibitem{Vizing1963}
V. G. Vizing, The Cartesian product of graphs. Vy\v{c}isl. Sistemy, 9(1963), 30-43.

\bibitem{Volkmann2018}
L. Volkmann, Double Roman domination and domatic numbers of graphs, Commun. Comb. Optim., 3(2018), 71-77.


\bibitem{West}
D. B. West, Introduction to graph theory, Upper Saddle River, NJ:Prentice Hall,1996.

\bibitem{Xi}
C. Q. Xi and J. Yue, The restrained double Roman domination of graphs, submitted.

\bibitem{strong}
I. G. Yero and J. A. Rodr\'{i}guez-Vel\'{a}zquez. Roman domination in Cartesian product graphs and strong product graphs, Appl. Anal. Discrete Math., 7(2)2013, 262-274.




\end{thebibliography}
\end{document}